\documentclass[reqno,10pt]{amsart} 
\usepackage{amssymb, amscd, verbatim, amsmath, MnSymbol, stmaryrd}
\usepackage{color}

\numberwithin{equation}{section}
\newtheorem{theorem}{Theorem}[section]
\newtheorem{lemma}[theorem]{Lemma}
\newtheorem{corollary}[theorem]{Corollary}

\newtheorem{remark}[theorem]{Remark}

\theoremstyle{definition}

\def\tr{\Delta}
\def\bea{\begin{eqnarray*}}
\def\eea{\end{eqnarray*}}
\def\be{\begin{eqnarray}}
\def\ee{\end{eqnarray}}
\def\a{\alpha}

\def\n{\nabla} 
\def\e{\epsilon} 
\def\o{\omega}
\def\d{\delta} 
\def\vp{\varphi}
\def\o{\omega}
\def\dis{\displaystyle}
\def\hess{{\rm Hess}}
\def\ric{{\rm Ric}}
\def\scal{{\rm Scal}}

\begin{document}


 \title{Structures of  compact static vacuum   spaces and positive isotropic curvature}

\author{Seungsu Hwang}
\address{Department of Mathematics\\ Chung-Ang University\\
84 HeukSeok-ro DongJak-gu \\ Seoul 06974, Republic of Korea.
}
\email{seungsu@cau.ac.kr}

 \author{Gabjin Yun$^*$}
\address{Department of Mathematics\\  Myong Ji University\\
116 Myongji-ro Cheoin-gu\\ Yongin, Gyeonggi 17058, Republic of Korea. }
\email{gabjin@mju.ac.kr} 

\thanks{$*$ Corresponding author}

   

 \maketitle

 \begin{abstract}
In this paper, we study  geometric structures of compact static vacuum spaces  and function theoretic properties for the potential functions satisfying the static vacuum equation. In particular, we investigate some geometric conditions  that  static vacuum spaces    are warped product  and   Bach flat.  As an application, we prove that if a triple $(M^n, g, f), n \ge 4$, is a compact  static vacuum space satisfying $\omega:=df \wedge i_{\n f}{\mathring {\rm Ric}} = 0$, then $M$ is either isometric to a round sphere  or a warped product of a circle with a compact Einstein manifold of positive Ricci curvature, up to finite cover.
Furthermore, if $(M, g)$ has  positive isotropic curvature, then
  $M$ is either isometric to a round sphere or a product ${\Bbb S}^1 \times {\Bbb S}^{n-1}$.
  \end{abstract}


\setlength{\baselineskip}{16pt}

\section{Introduction}

An $n$-dimensional complete Riemannian manifold $(M,g)$ is said to be a {\it static space} with 
a perfect fluid if there exists a smooth non-trivial function $f$ on $M$ satisfying 
\be 
\hess f -\left({\rm Ric} -\frac {{\rm Scal}}{n-1}g\right) f =\frac 1n \left( \frac {\scal}{n-1}\,f +\tr f\right)g,
\label{eqn1-1}\ee
where $\hess f$ is the Hessian of $f$,  $\ric$ is the Ricci tensor of $g$ with its scalar curvature 
$\scal$,  and $\tr f$ is the (negative) Laplacian of $f$. In particular, if 
\be
\tr f =-\frac {\scal}{n-1}\, f, \label{eqn2018-9-5-1}
\ee
$(M,g)$ is said to be a {\it static vacuum space}. In this case, the equation (\ref{eqn1-1}) reduces to 
\be 
\hess f - \left(\ric -\frac {\scal}{n-1}g\right) f =0, \label{static1}
\ee
which is called the {\it static vacuum equation}. Using the trace-less Ricci tensor $\mathring{\ric} = \ric - \frac{\scal}{n}g$,
the static vacuum equation (\ref{static1}) can be written as
\be
f {\mathring{\ric}} = \hess f + \frac{\scal}{n(n-1)} fg.\label{eqn2020-1-7-11}
\ee
It is easy to see that  if a smooth function $f$ on $M$ is a solution of the 
static vacuum equation (\ref{static1}), then $f$ 
 is an element of the kernel space, $\ker s_g'^*$, of  the operator $s_g'^*$, where
 $s_g'^*$ is the $L^2$-adjoint operator of the linearized scalar curvature $s_g'$
 with respect to the metric $g$ which is given by
\bea
s_g'^*(\phi)= \hess \, \phi-(\tr \phi)g - \phi\, {\ric}  
\eea
for any smooth function $\phi$ on $M$ (cf. \cite{Be}).

By taking the divergence of (\ref{static1}), we have $\frac 12 f\, d\scal=0$,  which shows the scalar curvature  of $(M, g)$ is constant since $f$ is a nontrivial function.  When $M$ is compact,  it is  known \cite{Bour} that a compact static vacuum space is either isometric to a Ricci-flat  manifold with $\ker s_g'^* ={\Bbb R}\cdot 1$ and $\scal=0$, or the scalar curvature $\scal$ is a  strictly positive constant and $\frac{\scal}{n-1}$ is an eigenvalue of the Laplacian. In particular, if $f$ is s nontrivial solution of the static vacuum equation, then the scalar curvature is  a positive constant.  
In case of compact static vacuum spaces, if the potential function $f$ is nonnegative, it follows from (\ref{eqn2018-9-5-1}) that $f$ should be constant   which contradicts non-triviality of $f$. Thus, we may assume that $\min_M f <0$. With this assumption,  it is well-known that each connected component of the zero set $f^{-1}(0)$ of the potential function $f$ is  a totally geodesic hypersurface. 
Also, it turns out  \cite{corv} that the warped product manifold  $(M\times_{f} {\mathbb R}, g\pm f^2 dt^2)$ is Einstein when $f \in \ker s_g'^*$ .

When  the  static space  $(M, g)$ is locally conformally flat,  Kobayashi and Obata  \cite{k-o} 
 proved that, around the hypersurface  $f^{-1}(c)$ for a regular value $c$, 
the metric $g$ is isometric to a warped product metric of constant curvature.
For static vacuum spaces, Fischer and Marsden   \cite{f-m} conjectured that 
only compact static vacuum
 space is a standard sphere, but it is known that besides flat torus
 and round spheres, product ${\Bbb S}^1\times {\Bbb S}^{n-1}$ and warped product 
${\Bbb S}^1\times_\xi {\Bbb S}^{n-1}$ are also compact static vacuum spaces
 (for references, see \cite{eji},  \cite{kob}, \cite{laf},  \cite{shen}) for locally conformally flat case.  
 Furthermore, it is known  \cite{QY} that  
 ${\Bbb S}^1 (\frac{1}{\sqrt{n-2}}) \times \Sigma^{n-1}$ for Einstein manifold $\Sigma$ with
 scalar curvature $(n-1)(n-2)$ is a compact static vacuum space, which may not be  locally 
conformally flat.

Some rigidity results related to static vacuum spaces have been found. For example,  Kim and Shin  \cite{ks} proved a local 
classification of  four-dimensional  static vacuum spaces with harmonic curvature. We say that a Riemannian manifold $(M, g)$ has harmonic curvature if ${\rm div}\,  R = 0$, or equivalently, that the Ricci tensor  is a Codazzi tensor.
On the other hand, Qing and Yuan showed  \cite{QY} that, if $(M,g,f)$ is a Bach-flat static vacuum 
 space with compact level sets of $f$,  then it is either Ricci-flat, isometric to 
${\Bbb S}^n$,  ${\Bbb H}^n$, or  a warped product space. Recall that a Riemannian manifold $(M^n, g), n \ge 4$, is said to be Bach-flat if
 the Bach tensor  (see Section 2 for definition)  vanishes.
 The Bach tensor  discussed first by Bach in \cite{bach} is deeply related  to general relativity and conformal geometry (cf. \cite{lis}), and
in  dimension $n=4$, it is well known \cite{Be} that the Bach tensor 
is conformally invariant and arises as a gradient of the total Weyl curvature  functional. In this article, we prove that a compact static vacuum space whose Ricci curvature is vanishing in the direction of gradient of potential function must be Bach-flat.

 \begin{theorem} \label{lem2025-3-11-1}
Let $(M^n,  g, f), n\ge 4, $ be a compact static vacuum space of dimension $n$  satisfying  $\o = df\wedge i_{\n f}\mathring{\rm Ric} =0$.  
 Then  $(M, g)$ is Bach-flat. Here, $i_{\n f}$ denotes the  interior product.
 \end{theorem}

As an application of Theorem~\ref{lem2025-3-11-1} together with Qing and Yuan's result, we can show the following rigidity result for compact
vacuum spaces.

 \begin{theorem} \label{lem2020-1-7-16-02}
Let $(M^n, g,f)$ be a compact static vacuum space of dimension $n$  satisfying  $\o = df\wedge i_{\n f}\mathring{\rm Ric} =0$.  
 Then  up to finite cover, either $M$ is isometric to ${\Bbb S}^{n}$ or a product ${\Bbb S}^1 \times {\Sigma}^{n-1}$, where $\Sigma$ is a compact Einstein manifold of positive Ricci curvature.
\end{theorem}

We would like to mention \cite{ych} that if a static vacuum space $(M ,g, f)$ has harmonic curvature, then $\o$ must  vanish.
In other words, our results on compact static vacuum spaces can be considered as a generalization of this result.

In the  last section of this paper, we study  static vacuum spaces having positive isotropic curvature.
Let $(M^n, g)$ be an $n$-dimensional Riemannian manifold with $n \ge 4$. 
The Riemannian metric $g = \langle \,\, ,\, \rangle$ can be extended  either
to a {\it complex bilinear form} $(\,\, , \, )$ or a {\it Hermitian inner product} 
$\llangle\,\,, \, \rrangle$ on each complexified tangent space  $T_pM \otimes {\Bbb C}$ for $p \in M$.
A complex $2$-plane $\sigma \subset T_pM\otimes {\Bbb C}$ is {\it totally isotropic}
 if $ (Z, Z) = 0$ for any  $Z \in \sigma$. For any $2$-plane $\sigma \subset T_pM\otimes {\Bbb C}$, we can define the
 complex sectional curvature of $\sigma$ with respect to $\llangle  \,  , \, \rrangle$ by
 \be
 {\rm K}_{\Bbb C}(\sigma) = \llangle {\mathcal R}(Z \wedge W), Z\wedge W\rrangle,
 \label{eqn2019-5-17-1}
 \ee
 where $\mathcal R : \Lambda^2 T_pM \to \Lambda^2 T_pM$ is the curvature operator and 
 $\{Z, W\}$ is a unitary basis for $\sigma$ with respect to $\llangle \,  , \, \rrangle$.
 
  A Riemannian $n$-manifold $(M^n, g)$ is said to have 
 {\it positive isotropic curvature} {\rm(PIC in short)} if
  the compex sectional curvature on isotropic planes is positive, that is,
 for any totally isotropic $2$-plane $ \sigma \subset T_pM\otimes {\Bbb C}$,
 \be
 {\rm K}_{\Bbb C}(\sigma) >0. \label{eqn2018-4-22-1}
 \ee
 In view of (\ref{eqn2019-5-17-1}), the condition (\ref{eqn2018-4-22-1})
  can be formulated in purely real terms.  For orthonormal  vectors
  $ \{e_1 , e_2, e_3,  e_4\}$   in $T_pM$ with 
  $\sigma = {\rm span}\{e_1+ie_2, e_3 +ie_4\}$, letting
$$
Z = e_1 + ie_2, \quad W= e_3+ie_4,
$$
we have
 \bea
2 {\rm K}_{\Bbb C}(\sigma) =
 \llangle  {\mathcal R}(Z\wedge W), Z\wedge W\rrangle
 = R_{1313} +R_{1414} + R_{2323} + R_{2424} - 2R_{1234}.
 \eea
 Hence $(M^n, g)$ has {PIC} if and only if for any four orthonormal frame 
 $\{e_1, e_2, e_3, e_4\}$,
 \bea
 R_{1313} +R_{1414} + R_{2323} + R_{2424} >  2R_{1234}. 
 \eea

 If $(M, g)$ has positive curvature operator, then it has PIC  \cite{mm88}. 
Thus, a standard sphere $({\Bbb S}^n, g_0)$ has PIC. 
Also, if the sectional curvature of $(M, g)$  is pointwise strictly quater-pinched, then  
$(M, g)$ has PIC \cite{mm88}.
  It is well-known that the product metric on ${\Bbb S}^{1}\times {\Bbb S}^{n-1}$ has also
  PIC and  the connected sum of manifolds with PIC admits a PIC metric \cite{m-w}.
  For existence of Riemannian metrics with PIC on compact manifolds which  fiber
  over the circle, see \cite{lab}.
  On the other hand, PIC implies that $g$ has positive scalar curvature \cite{m-w}.
Recently, there are many known results on the structure of Riemannian manifolds 
with positive isotropic curvature (\cite{bren}, \cite{c-hu}, \cite{ctz}, \cite{fra}, \cite{f-w}, \cite{sea}, \cite{ses} and references are therein).

Related to static vacuum spaces having PIC, we  have the following.

\begin{theorem}\label{cor2023-4-22-1-02}
Let $(M^n, g, f), n \ge 4, $  be a $n$-dimensional compact static vacuum space of positive isotropic curvature such that $\o = df\wedge i_{\n f} \mathring{\rm Ric} = 0$. Then $M$ is isometric to a sphere ${\Bbb S}^n$  or a  product ${\Bbb S}^1\times {\Bbb S}^{n-1}$,  up to finite cover.
\end{theorem}

Throughout this paper, we follow conventions in \cite{Be} on curvatures and geometric differential operators like divergence 
except  the Laplace operator. And, hereafter,  for convenience and simplicity, we denote curvatures $\ric,  {\mathring \ric},  \scal$, and   $\hess f$  by $r, z, s$, and $Ddf$, respectively, if there are no ambiguities.

\section{Preliminaries}

In this section, we  introduce a structural $3$-tensor $T$  on a static vacuum space $(M^n, g, f)$ using the potential function $f$ and the trace-less Ricci tensor $z$. It has two terms consisting of wedge product of 1-forms and symmetric 2-tensors. Extracting 1-forms from each term contained in this tensor,   we will make a 2-form $\omega$ by using wedge product. This $2$-form together with $T$ plays important roles in the proofs of our main results.

Let $(M, g, f)$  be a static vacuum space satisfying (\ref{static1}) or equivalently (\ref {eqn2020-1-7-11}).
We  define  a $3$-tensor $T$ on $(M, g, f)$  by  
\bea
T= \frac 1{n-2}\, df  \wedge z +\frac 1{(n-1)(n-2)}\, i_{\nabla f }z \wedge g. 
\eea
As mentioned in Theorem~\ref{lem2025-3-11-1}, $i_{\n f}$ denotes the usual interior product to the first factor defined by
$i_{\n f}z(X)  = z(\n f, X)$ for any vector $X$, and for a $1$-form $\eta$ and a symmetric $2$-form $b$, $\eta\wedge b$ is defined as
$$
\eta\wedge b(X, Y, Z) = \eta(X) b(Y, Z) - \eta(Y)b(X, Z)
$$
for any vector fields $X, Y$ and $Z$. From definition, one can see that  the cyclic summation of $T_{ijk}$ is always vanishing. 
In \cite{QY}, Qing and Yuan used a very similar tensor as the tensor $T$ to classify static vacuums spaces with Bach flat structure. 
The tensor $T$ also looks very similar as in \cite{CC} (cf. \cite{k-o}), where the authors used this similar tensor to  classify complete Bach flat gradient shrinking Ricci solitons.
The norm of $T$ is given in the following form (\cite{h-y},  \cite{GH1}).

\begin{lemma}\label{lem2019-5-28-5}
Let $(g, f)$ be a non-trivial solution of the static vacuum equation, Then, with $N = \frac{\n f}{|\n f|}$,
$$ 
|T|^2= \frac 2{(n-2)^2}|\nabla f|^2 \left( |z|^2-\frac n{n-1}|i_Nz|^2\right).
$$
\end{lemma}

The tensor $T$ is deeply related to the Bach tensor and the Cotton tensor. First of all, we here present the Bach tensor and the Cotton tensor for self-containedness. 
Let $(M^n, g)$ be a Riemannian manifold of dimension $n$ with the Levi-Civita  connection $D$. We start with a differential operator acting on the space  of symmetric $2$-tensors. Let $b$ be a symmetric $2$-tensor on $M$. The covariant differential $d^Db$ is defined by 
$$ 
d^Db(X, Y, Z) = D_Xb(Y, Z) - D_Yb(X, Z) 
$$
 for any vectors $X, Y$ and $Z$. The {\it Cotton tensor} 
 $C \in\Gamma(\Lambda^2 M \otimes T^*M)$ is defined by 
 \bea
 C = d^D \left(r - \frac{s}{2(n-1)} g\right)  = d^Dr - \frac{1}{2(n-1)} ds \wedge g. 
 \eea 
  It is well-known that, for $n=3$, 
$C=0$ if and only if $(M^3, g)$ is locally conformally flat.
 Moreover, for $n \ge 4$, the vanishing of Weyl
tensor $\mathcal W$ implies the vanishing of the Cotton tensor $C$, while
$C=0$ corresponds to the Weyl tensor being harmonic, i.e, $\d \mathcal W
= 0$, where $\d$ denotes the negative divergence operator. In fact, we have  the following identity  \cite{Be}: for a
Riemannian $n$-manifold $(M^n, g)$, the Weyl tensor $\mathcal W$ satisfies 
\be 
\d \mathcal W = - \frac{n-3}{n-2}d^D \left(r - \frac{s}{2(n-1)} g\right) 
= - \frac{n-3}{n-2} C\label{eqn2016-12-3-16} 
\ee 
under the following identification
$$ 
\Gamma(T^*M\otimes \Lambda^2M) \equiv \Gamma(\Lambda^2M \otimes T^*M). 
$$

\vspace{.13in}
  The {\it Bach tensor} $B$ on a Riemannian manifold $(M^n, g)$, $n \ge 4$,
  is  defined by 
\bea
B = \frac{1}{n-3}\d^D \d \mathcal W + \frac{1}{n-2}{\mathring {\mathcal W}}r.\label{eqn2018-2-11-2} 
\eea
Here $\d^D$ is the adjoint operator of $d^D$. 
For a symmetric $2$-tensor $b$ on a Riemannian manifold $(M^n, g)$, 
we define ${\mathring {\mathcal W}}b$ by
$$ 
{\mathring {\mathcal W}}b(X, Y) = \sum_{i=1}^n b(\mathcal W (X, E_i)Y, E_i) 
= \sum_{i, j=1}^n \mathcal W (X, E_i, Y, E_j)b(E_i, E_j) 
$$ 
for any local frame $\{E_i\}$.
From (\ref{eqn2016-12-3-16}), we have 
\bea
C = - \frac{n-2}{n-3} \d \mathcal W\quad \mbox{and}\quad \d C
 = - \frac{n-2}{n-3}\d^D \d \mathcal W.\label{candw} 
\eea
Thus, one can see that the Bach tensor satisfies 
\bea
B = \frac{1}{n-2}\left(-\d C + {\mathring {\mathcal W}}z\right). 
\eea
Finally, we have the following identity \cite{h-y}:
\be
f C = {\tilde i}_{\n f} \mathcal W - (n-1)T,\label{eqn2017-6-12-10-1}
\ee
where   $\tilde i_{\n f}$ denotes the interior product to the last component so that   
${\tilde i}_{\n f} \mathcal W(X, Y, Z) = \mathcal W(X, Y, Z, \n f)$. 

The following result can be obtained from a  result of Qing and Yuan in \cite{QY}.

\begin{theorem}\label{thm2018-1-20-11}
 Let $(M^n, g, f)$ be a compact static vacuum space with nontrivial potential function $f$.  If $T = 0$, then, up to a finite cover and appropriate scaling,    $M$ is either isometric to a sphere ${\Bbb S}^n$, or  a warped product, ${\Bbb S}^1\times_\xi \Sigma^{n-1}$, 
 $g = dt^2 + \xi(t)^2 g_\Sigma$,   of a circle and a compact Einstein manifold $\Sigma$ with positive scalar  curvature.
\end{theorem}
\begin{proof} 
First of all, note that $(M, g)$ has   positive constant scalar curvature. It is proved in \cite{h-y} that a static vacuum space  $(M^n, g, f)$ with $T=0$ is Bach-flat. The conclusion follows from  a result due to Qing and Yuan \cite{QY}. In dimension $3$, we also have $C=0$, which implies the same result.
\end{proof}

Next, we will introduce a special type of $2$-form on a static vacuum space. Let $(g, f)$ be a non-trivial solution of the static vacuum equation on an $n$-dimensional compact manifold $M$.
Considering $i_{\nabla f} z$ as a $1$-form,  we define a $2$-form  $\omega$ by 
$$
\omega = df \wedge i_{\nabla f}z.
$$ 

\begin{lemma}\label{lem2018-4-30-20}
 We have
 \be
 \omega = (n-1) {\tilde i}_{\n f} T = - f {\tilde i}_{\n f}C.\label{keyeqn}
 \ee
 Here ${\tilde i}_{\n f} T(X, Y) = T(X, Y, \n f)$ and ${\tilde i}_{\n f}C$ is defined similarly.
 \end{lemma}
  \begin{proof}
It follows from the  definition of $T$ that
 $$
 (n-2) T(X, Y, \nabla f)
=\frac {n-2}{n-1}\, df\wedge i_{\nabla f} z(X,Y)=\frac {n-2}{n-1}\, \omega (X,Y)
$$
for vectors $X$ and $Y$. The second equality follows from (\ref{eqn2017-6-12-10-1}). 
 \end{proof}
 
 \begin{lemma}\label{simple}
 Let  $\{E_1, E_2,  \cdots, E_n\}$ be a local frame with $E_1=N = \frac{\n f}{|\n f|}$.
Then
$$
\omega =0 \quad \mbox{if and only if}\quad \tilde{i}_{\nabla f}C(N, E_j)=0\quad (j \ge 2).
$$
\end{lemma}
\begin{proof}
It follows from the definition of $\omega$ that
\bea
\omega (E_j, E_k)=0\quad\mbox{for all $j, k \ge 2$}, 
\eea
which shows, by Lemma~\ref{lem2018-4-30-20}, 
\bea
\tilde{i}_{\nabla f}C(E_j, E_k)=0\label{simple-2}
\eea
for $2\leq j,k\leq n$.  Therefore, it is easy to see that
$$
\omega =0 \quad \mbox{if and only if}\quad \tilde{i}_{\nabla f}C(N, E_j)=0
$$
for $2\leq j\leq n$.
\end{proof}

\begin{lemma} \label{lem191} 
As a $2$-form, we have the following 
$$  
\tilde{i}_{\nabla f}C= di_{\nabla f}z.
$$
\end{lemma}
\begin{proof} 
Choose a local frame $\{E_i\}$ which is normal at a point $p \in M$, 
and let $\{\theta ^i\}$ be its dual coframe so that $d\theta^i\vert_{p}=0$.  
Since $i_{\nabla f}z=\sum_{l,k=1}^n f_l z_{lk}\theta ^k$ with $E_l(f) = f_l$ and
$z(E_l, E_k) = z_{lk}$, by (\ref{eqn2020-1-7-11}), we have
\bea
di_{\n f}z &=& 
\sum_{j,k} \sum_{l} (f_{lj} z_{lk} + f_l z_{lk;j}) \theta^j \wedge \theta^k \\
&=&
\sum_{j<k} \sum_{l} \left\{(f_{lj} z_{lk} -f_{lk} z_{lj}) +  f_l (z_{lk;j} - z_{lj;k})\right\} 
\theta^j \wedge \theta^k \\
&=&
\sum_{j<k} \sum_{l} \left[\left\{\left( f z_{lj} -\frac{sf\,\delta_{lj} }{n(n-1)}\right) z_{lk}
- \left( f z_{lk} -\frac{sf\,\delta_{lk}}{n(n-1)} \right) z_{lj}\right\}  \right] \theta^j \wedge \theta^k \\
& &+ \sum_{j<k} \sum_{l} f_l C_{jkl}  \theta^j \wedge \theta^k\\
&=&
 \sum_{j<k} \sum_{l} f_l C_{jkl}  \theta^j \wedge \theta^k =
 {\tilde i}_{\n f}C.
 \eea
\end{proof}

\begin{lemma}\label{closedform}
$\o$ is a closed $2$-form, i.e.,
$d\o = 0$.
\end{lemma}
\begin{proof}
Choose a local frame $\{E_i\}$ with $E_1 = N = {\n f}/{|\n f|}$, and let $\{\theta^i\}$ be its dual coframe. 
Then, by Lemma~\ref{simple} and Lemma~\ref{lem191} 
$$
di_{\n f}z = 
\sum_{j<k} \sum_{l} f_l C_{jkl}  \theta^j \wedge \theta^k 
= \sum_{j<k}  |\n f| C_{jk1}  \theta^j \wedge \theta^k 
=
 |\n f| \sum_{k=2}^n   C_{1k1}  \theta^1 \wedge \theta^k. 
$$
Thus, by taking the exterior derivative of $\omega$ in (\ref{keyeqn}), we have
$$
d\o = - df \wedge d i_{\n f}z = - |\n f| \theta^1 \wedge
\left( |\n f| \sum_{k=2}^n   C_{1k1}  \theta^1 \wedge \theta^k\right) =0.
$$
\end{proof}

\section{static vacuum spaces with $\o =0$}

In this section, we give some geometric structural properties for static vacuum spaces such that the 2-form  $\o = df \wedge i_{\n f} z$
is vanishing. If $\o =0$, then one can see that
\bea
z(\nabla f, X)=0 
\eea
for any vector field $X$ orthogonal to $\n f$, and we may write 
$$
i_{\nabla f}z =\a df,\quad \mbox{where}  \,\,  \a = z(N, N) \,\,\mbox{with}\,\, N = \frac{\n f}{|\n f|}
$$
as a $1$-form. Note that the function $\a$ is well-defined only on the set $M \setminus {\rm Crit}(f)$, where ${\rm Crit}(f)$ is the set of
all critical points of $f$. However since $|\a| \le |z|$, $\a$ can be extended to a $C^0$ function on the whole $M$. 
See \cite{ych} for more details.

\begin{lemma} \label{Cotz}
Let $(g, f)$ be a non-trivial solution of the static vacuum equation satisfying  $\o=0$.  Then we have 
$$
D_NN = 0, \quad \tilde{i}_{\nabla f}C=0 \quad \mbox{and}\quad \tilde{i}_{\nabla f}T=0.$$
In particular,  we have
\be 
i_{\nabla f}z =\a \, df\label{eqn1}
\ee
as a $1$-form, and so
\be
f|z|^2= \nabla f(\a)  -\frac {sf}{n-1}\a. \label{eqn2}
\ee
\end{lemma}
\begin{proof}
Since $N(|\n f|) = f\a - \frac{sf}{n(n-1)}$, it is easy to see that $D_NN=0$. The others  follow directly from Lemma~\ref{lem2018-4-30-20}. For (\ref{eqn2}), 
we note from the static vacuum equation that
$$
\d (i_{\n f}z) = -f|z|^2.
$$
So, we can obtain (\ref{eqn2}) by taking the divergence of (\ref{eqn1}).
\end{proof}

One of basic observations for static vacuum spaces $(M^n, g, f)$ satisfying $\o = df \wedge i_{\n f}z = 0$ is the following. 

\begin{lemma}\label{lem2020-8-23-2}
The functions $|\n f|, \a$ and $|z|^2$ are constants on each (connected component of) 
the level hypersurfac $f^{-1}(c)$ of $f$.
\end{lemma}
\begin{proof}
Let $X$ be a tangent vector on $f^{-1}(c)$ so that $z(X, \n f ) = 0$ by our assumption.
 Therefore, from (\ref{eqn2020-1-7-11}), we 
have,
$$
\frac{1}{2} X(|\n f|^2) = Ddf(\n f, X) = z(\n f, X) - \frac{sf}{n(n-1)} g(\n f, X) = 0.
$$ 
Since $D_NN=0$ by Lemma~\ref{Cotz}, we have
$g(D_NX, N) = - g(X, D_NN) = 0$ and $D_N\n f = - |\n f|N\left(\frac{1}{|\n f|}\right) \n f$. So,
\be
D_N z(X, \n f) = - z(D_N X, \n f)-z(X, D_N \n f) = 0.\label{eqn2020-8-26-1}
\ee
Since $\tilde{i}_{\nabla f}C=0$ by Lemma~\ref{Cotz} again, we have,
$$
0=C(X, N, \nabla f)=D_X z(N, \nabla f)-D_{N}z(X, \nabla f) = D_X z(N, \nabla f)=|\nabla f|X(\a),
$$
implying that $\a $ is a constant on $f^{-1}(c)$.

The property $D_NN=0$  also implies $ [X, N]$ is orthogonal to $\n f$ and, therefore,
$$
X(N(\alpha)) = N(X(\alpha)) - [X, N](\alpha) = 0.
 $$
Since $\n \alpha = N(\alpha)N$, it shows $Dd\a(X, \n f) = 0$, and from (\ref{eqn2}), we have $X(|z|^2) = 0$.
Consequently, we conclude that $|z|^2$ is a constant along each level set $f^{-1}(c)$ of $f$.
\end{proof}

\begin{lemma}\label{cor191-1}
Let $(g,f)$ be a non-trivial solution of the static vacuum equation  on an $n$-dimensional compact manifold $M$ satisfying $\o=0$. Then 
there are no critical points of $f$ other than its maximum and minimum points of $f$.
\end{lemma}
\begin{proof}
Note that $|\nabla f|^2$ is constant on each level sets of $f$.  In fact, let $X$ be any vector field orthogonal to $\n f$. It follows from the static equation (\ref{eqn2020-1-7-11}) together with our assumption that
$$
\frac12 X(|\n f|^2) = Ddf(\n f, X) = fz(\n f, X) - \frac{sf}{n(n-1)}g(\n f, X) = 0.
$$
From the Bochner-Weitzenb\"ock formula (cf. \cite{pl1}, \cite{wu})  together with (\ref{eqn2020-1-7-11}),   
we have
\be
\frac 12 \tr |\nabla f|^2-\frac 12 \frac {\nabla f(|\nabla f|^2)}{f}  =|Ddf|^2 \ge 0. \label{eqn2020-1-14}
\ee
By the maximum principle, the function $|\nabla f|^2$ cannot have its 
local maximum in  $M_0:=\{x\in M \, |\, f(x) < 0\}$ nor in
 $M^0:=\{x\in M \, |\, f(x) > 0\}$. In other words, $|\n f|^2$ may attain its local
 maximum only on the set $f^{-1}(0)$.
 
Let $p \in M$ be a  critical point of $f$  with $f(p) = c$
other than  minimum or maximum points 
of $f$ so that $\n f(p) = 0$. Then the function $|\n f|^2$ must have a local maximum between $f^{-1}({\rm min})$ and $f^{-1}(c)$, or
between $f^{-1}(c)$ and $f^{-1}({\rm max})$.
This contradicts the argument mentioned above.
\end{proof}
 
Note that   if minimum set or  maximum set of the potential function $f$ contains more than two points, or is a submanifold of codimension $>1$, the function $f$ must contains its critical points in $M$ other than  minimum and maximum sets of $f$.
Thus, Lemma~\ref{cor191-1} shows that minimum set of $f$ consists of only a single point or a hypersurface in $M$, and the same property holds for maximum set of $f$.
The following result shows that the structure of minimum set or maximum set of the potential function  is either a single point or a totally geodesic stable minimal hypersurface. Since  proof for this is long, but routine and not hard, we will give a proof  in Appendix. 
  
\begin{theorem}\label{thm2019-12-23-1}
Let $(g,f)$ be a non-trivial solution of the static vacuum equation on an $n$-dimensional compact manifold $M$ satisfying  $\o=0$.  Let ${\min_M f = a}$.
Then either $f^{-1}(a)$  contains only a single point, or $f^{-1}(a)$ is a totally geodesic stable minimal hypersurface of $M$.
The same property also holds for  $f^{-1}(b)$ with ${\max_M f = b}$.
Furthermore, if $f^{-1}(a)$ is a single point, then so is $f^{-1}(b)$ and vice versa.
In case that $f^{-1}(a)$ is a single point, every level set $f^{-1}(t)$ except $f^{-1}(a)$ and $f^{-1}(b)$ is
a hypersurface and is homotopically a sphere ${\Bbb S}^{n-1}$.
\end{theorem}

From Lemma~\ref{cor191-1} and Theorem~\ref{thm2019-12-23-1}, we have the following.

\begin{theorem}\label{thm2020-6-18}
Let $(g,f)$ be a non-trivial solution of the static vacuum equation on an $n$-dimensional compact manifold $M$ positive scalar curvature
satisfying  $\o=0$. Let  ${\min_M f = a}$. Suppose that $f^{-1}(a)$ is not a single point. Then  $f^{-1}(a)$ is a compact Einstein manifold of positive Ricci curvature.
\end{theorem}
\begin{proof}
Let $\Sigma = f^{-1}(a)$ and we may assume that $a<0$ from the maximum principle together with (\ref{eqn2020-1-14}).
It follows  from Lemma~\ref{cor191-1} and Theorem~\ref{thm2019-12-23-1}
 that $\Sigma$ is a connected, totally geodesic stable minimal hypersurface of $M$. In particular, from (\ref{static1}) together with $a<0$, $\Sigma$
 satisfies ${\rm Ric}_\Sigma = \frac{s}{n-1}g_\Sigma$, which shows $(\Sigma, g_\Sigma)$ is an Einstein manifold of positive Ricci curvature.
 \end{proof}

The following property is well-known.

\begin{lemma}\label{thm2020-1-11-1}
Let $(g,f)$ be a non-trivial solution of the static vacuum equation on an $n$-dimensional compact manifold $M$. Then any connected component of $f^{-1}(0)$ is a totally geodesic hypersurface of $M$.
\end{lemma}

In case that $\o$ is vanishing for a static vacuum space $(M^n, g, f)$, the tensor $T$ has the following proeprties.

 \begin{lemma}\label{lem7}
Suppose $\o=0$. Then 
\begin{itemize}
\item[(1)] for vectors $X, Y$ orthogonal to $\n f$,
$$
i_{\n f}T(X, Y) = \frac{|\n f|^2}{n-2}\left(z+\frac{\a}{n-1}g\right)(X, Y)
$$
\item[(2)] $i_{\n f}T(\n f, X) = i_{\n f}T(X, \n f) = 0$ for any vector $X$.
\end{itemize}
\end{lemma}
\begin{proof}
If $\o= 0$, then $i_{\n f}z = \a df$ and so
$$
T = \frac{1}{n-2} df\wedge \left(z+\frac{\a}{n-1}g\right).
$$
\end{proof}

For the curvature tensor $R$ with $N = \n f /|\n f|$, $R_N$ is defined  as follows 
$$
R_N(X, Y) = R(X, N, Y, N)
$$
for any vector fields $X$ and $Y$. For the Weyl curvature tensor $\mathcal W$, 
$\mathcal W_N$ is
similarly defined.  It follows from (\ref{eqn2017-6-12-10-1}) that
 \be
 -|\n f|^2 \mathcal W_N = f i_{\n f}C + (n-1)i_{\n f}T.\label{eqn2019-5-28-1}
 \ee

\begin{lemma} \label{lem2018-8-15-4}
Let $(g, f)$ be a non-trivial solution of the static vacuum equation satisfying  $\o=0$.  Then
$$
\frac{s}{n(n-1)} g  = R_N +z +\frac{f}{|\n f|^2}i_{\n f}C+ \left(\frac{s}{n(n-1)} -\a \right) \frac{df}{|df|}\otimes \frac{df}{|df|}.
$$
\end{lemma}
\begin{proof}
 Let
 $$
 \Phi:=  \frac{s}{n(n-1)}g  - z - \frac{f}{|\n f|^2}i_{\n f}C.
 $$
 For vector fields $X, Y$ with $X \perp \n f$ and $ Y \perp \n f$, 
from the curvature decomposition 
$$
R = \frac{s}{2n(n-1)}g\owedge g + \frac{1}{n-2}z\owedge g + \mathcal W
$$
we can obtain
 \bea
 R_N(X, Y) = \frac{s}{n(n-1)}g(X, Y) + \frac{1}{n-2} z(X, Y) + \frac{\a}{n-2} g(X, Y)
 +\mathcal W_N(X, Y).
 \eea
Since, by Lemma~\ref{lem7} together with (\ref{eqn2019-5-28-1})
 \bea
 \mathcal W_N(X, Y) =
 - \frac{f}{|\n f|^2}i_{\n f}C(X, Y) - \frac{n-1}{n-2}z(X, Y) - \frac{\a}{n-2}g(X, Y),
 \eea
we have
\bea
R_N(X, Y) = \Phi(X, Y). 
\eea
Now, let $X$ and $Y$  are arbitrary tangent vector fields. Then $X$ and $Y$ can be decomposed into
  $$
  X = X_1 + \langle X, N\rangle N,\quad Y = Y_1 +\langle Y, N\rangle N
  $$
  with $\langle X_1, N\rangle =0 =\langle Y_1, N\rangle.$
  Thus,
  \bea
  R_N(X, Y)    &=& 
  R_N(X_1, Y_1)= \Phi (X_1, Y_1)\\
&=&
\Phi(X, Y) - \langle X, N\rangle  \langle Y, N\rangle  \Phi(N, N)\\
 &=&
\Phi(X, Y) +   \left(\a -\frac{s}{n(n-1)} \right)  \frac{df}{|df|}\otimes \frac{df}{|df|}(X, Y).
   \eea   
\end{proof}

\section{ Bach flat static vacuum Spaces}

In this section, we prove that if a compact static vacuum space $(M^n,g, f), n \ge 4$  satisfies $\o = df \wedge i_{\n f}z = 0$, 
then $(M, g)$ is Bach-flat.
To do this,   we introduce a warped product metric involving
$\frac {df}{|\nabla f|}\otimes \frac {df}{|\nabla f|}$ as a fiber metric on each level
set  $f^{-1}(c)$. By investigating properties on this warped product metric,
and deducing its relations  to the given solution metric $g$, we prove the vanishing
of the tensor $T$. 
Consider a warped product metric $\bar{g}$ on $M$ by
\be
\bar{g}= \frac {df}{|\nabla f|}\otimes \frac {df}{|\nabla f|}+|\nabla f|^2 g_{\Sigma},
\label{eqn2019-8-27-1}
\ee
where $g_{\Sigma}$ is the restriction of $g$ to a level set  $\Sigma:= f^{-1}(c)$ for 
a regular value $c$ of $f$. 
Note that, from Theorem~\ref{thm2019-12-23-1}, 
the metric $\bar g$ is smooth on $M$ except only  $f^{-1}(a) \cup f^{-1}(b)$.

The following lemma shows that $\nabla f$ is a (closed) conformal  vector field with respect to the metric $\bar{g}$.

\begin{lemma}\label{lemt1}
Let $(g,f)$ be a non-trivial solution of the static vacuum equation on
 an $n$-dimensional compact manifold $M$ satisfying  $\o = 0$.  Then
$$\frac 12 {\mathcal L}_{\nabla f}\bar{g} =N(|\nabla f|)\bar{g}= \frac 1n 
(\bar{\tr }f)\, \bar{g}.
$$
Here $\mathcal L$ denotes the Lie derivative.
\end{lemma}
\begin{proof}
Note that, by (\ref{eqn2020-1-7-11}) we have
$$ 
\frac 12 {\mathcal L}_{\nabla f} g  =D_gdf = f z -\frac {sf}{n(n-1)}g.
$$
Moreover, the following identity holds:
\be 
N(|\nabla f|)= f\a -\frac {sf}{n(n-1)}.\label{eqnt3}
\ee
From these properties together with the definition of Lie derivative, we obtain
\bea
\frac{1}{2}{\mathcal L}_{\nabla f}(df \otimes df)(X,Y)&=& Ddf(X, \nabla f)df(Y)+df(X)Ddf(Y, \nabla f)\\
&=& 2\left( f\a -\frac {sf}{n(n-1)}\right) \, df\otimes df(X, Y).
\eea
Therefore,
\bea
 \frac 12 {\mathcal L}_{\nabla f}\left( \frac {df}{|\nabla f|} \otimes \frac {df}{|\nabla f|} \right) = N(|\nabla f|)  \frac {df}{|\nabla f|} \otimes \frac {df}{|\nabla f|}. 
\eea
Since
$$ 
\frac 12 {\mathcal L}_{\nabla f} (|\nabla f|^2 g_{\Sigma})
= \frac 12 \nabla f(|\nabla f|^2) g_\Sigma
= Ddf(\nabla f, \nabla f)g_\Sigma= N(|\nabla f|) |\nabla f|^2 g_\Sigma,
$$
we can conclude  that
$$ 
\frac 12 {\mathcal L}_{\nabla f}\bar{g}=\bar{D}df= N(|\nabla f|) \bar{g}.
$$
In particular, we have $\bar{\tr} f =n N(|\nabla f|) $.
\end{proof}

\begin{corollary}\label{lem2020-6-30-1}
Let $(g,f)$ be a non-trivial solution of the static vacuum equation on
 an $n$-dimensional compact manifold $M$ satisfying  $\o = 0$.  Then level hypersurfaces given by $f$ are homothetic to each other.
\end{corollary}
\begin{proof}
By Lemma~\ref{lemt1}, the gradient vector field $\n f$ is a conformal vector field with respect to the metric $\bar g$. Thus, as arguments in
\cite{tas2} or \cite{tas}, we can choose a local coordinate system $(u^i)$ in a neighborhood of any hypersurface $\Sigma:=f^{-1}(c)$, $a<c<b$  such that
\bea
\bar g = (du^1 )^2 + |\n f|^2 \eta_{ij}(u^2, \cdots, u^n) du^i \otimes du^j, 
\eea
where $du^1 = \frac{df}{|\n f|}$ and
 the functions $\eta_{ij}$ depend only on $u^2, \cdots, u^n$.
 Comparing this to (\ref{eqn2019-8-27-1}), we have
 $$
 g_{{\Sigma}} = \eta_{ij}(u^2, \cdots, u^n) du^i \otimes du^j.
 $$
 These show that level hypersurfaces are homothetic to each other and to $\Sigma$ with $\eta_{ij}(u^2, \cdots, u^n) du^i \otimes du^j$ as metric form.
\end{proof}

\begin{theorem} \label{lem2020-9-10-1}
Let $(g,f)$ be a non-trivial solution of the static vacuum equation on an $n$-dimensional compact manifold $M$ satisfying  $\o = 0$.  
 Then  $(M, g)$ is Bach-flat.
\end{theorem}
\begin{proof}
It follows from Lemma~\ref{lem2018-8-15-4} that
$$
\frac{s}{n(n-1)} g_{{\Sigma}}  = R_N + \left(z - \a \frac{df}{|df|}\otimes \frac{df}{|df|}\right)  +\frac{f}{|\n f|^2}i_{\n f}C
$$
for each level hypersurface $\Sigma$, and 
$$
 g  =  \frac{df}{|df|}\otimes \frac{df}{|df|} + \frac{n(n-1)}{s} \left[R_N + \left(z -  \a \frac{df}{|df|}\otimes \frac{df}{|df|}\right)
  +\frac{f}{|\n f|^2}i_{\n f}C\right].
  $$
Since $\o = 0$, the second term depends only on the $\n f$-direction by Lemma~\ref{lem2020-8-23-2}, and so,
by Corollary~\ref{lem2020-6-30-1}, the  metric $g$  can also be written as
\be
g= \frac {df}{|\nabla f|}\otimes \frac {df}{|\nabla f|}+ \xi(f)^2 g_0,\label{eqn2020-6-28-10}
\ee
where  $g_0 = g_{\Sigma_0}$ is the induced metric on $f^{-1}(0)$. Since
$$
\frac 12 {\mathcal L}_{\nabla f}\left( \frac {df}{|\nabla f|} \otimes \frac {df}{|\nabla f|} \right) = N(|\nabla f|)  \frac {df}{|\nabla f|} \otimes \frac {df}{|\nabla f|}
$$
and
$$ 
\frac 12 {\mathcal L}_{\nabla f}(\xi^2 g_0)=\xi\langle \nabla f, \nabla \xi\rangle g_0 
= \xi|\nabla f|^2 \frac {d\xi}{df}g_0,
$$
we have
\be
\frac 12 {\mathcal L}_{\nabla f}g=N(|\nabla f|)\frac {df}{|\nabla f|}\otimes \frac {df}{|\nabla f|}+\xi|\nabla f|^2 \frac{d\xi}{df}g_0.\label{eqnt5-1}
\ee
On the other hand, from the static vacuum equation together with (\ref{eqnt3}) and (\ref{eqn2020-6-28-10}),  we have
\bea
\frac 12 {\mathcal L}_{\nabla f}g 
&=&
 Ddf =fz -\frac {sf}{n(n-1)}g\\
&=& 
N(|\nabla f|) \frac {df}{|\nabla f|}\otimes \frac {df}{|\nabla f|} +fz-f\a \frac {df}{|\nabla f|}\otimes \frac {df}{|\nabla f|}  -\frac {sf}{n(n-1)}\xi^2 g_0.
\eea
Comparing this to (\ref{eqnt5-1}), we obtain
\be
\left( \xi|\nabla f|^2 \frac {d\xi}{df} +\frac {sf}{n(n-1)}\xi^2\right) g_0
=f\left(z-\a  \frac {df}{|\nabla f|}\otimes \frac {df}{|\nabla f|}\right)\label{eqnt7-1}
\ee
as $(n-1)\times (n-1)$ matrices.

Now, choosing a local frame  $\{E_1, E_2, \cdots, E_n\}$ with $E_1=N$,  we have
$$ 
\xi|\nabla f|^2 \frac {d\xi}{df}=fz(E_i, E_i)-\frac {sf}{n(n-1)}\xi^2
$$
for each $2\leq j\leq n$. Summing up these, we obtain
\be
(n-1)\xi|\nabla f|^2 \frac {d\xi}{df}= -f\a -\frac {sf}n\xi^2. \label{eqn2020-6-28-11-1}
\ee
Substituting this into (\ref{eqnt7-1}), we get
$$
 -\frac {\a}{n-1} g_0= z-\a  \frac {df}{|\nabla f|}\otimes \frac {df}{|\nabla f|}
 $$
 as $(n-1)\times (n-1)$ matrices.
This implies that, on each level hypersurface $f^{-1}(c)$, we have
$$ 
z(E_i, E_j)=-\frac {\a}{n-1}
$$
for $2\leq j\leq n$. Hence,
$$ 
|z|^2=\a^2+\frac {\a^2}{n-1}=\frac n{n-1}\a^2=\frac n{n-1}|i_Nz|^2,
$$
since $z(N, E_i)=0$ for $i\geq 2$. As a result, it follows  from 
Lemma~\ref{lem2019-5-28-5} that $T=0.$
So, $(M^n, g)$ is a Bach-flat static vacuum space. In case of dimension $n=3$, the vanishing of $T$ also implies the vanishing of the Cotton tensor $C$.
\end{proof}

Combining a result due to Qing and Yuan \cite{QY} with Theorem~\ref{lem2020-9-10-1}, we obtain the following.

\begin{theorem} \label{thm2019-6-22-1}
Let $(g,f)$ be a non-trivial solution of the static vacuum equation on an $n$-dimensional compact manifold $M$ satisfying  $\o = 0$.  
Let ${\min_M f = a}$. Suppose $f^{-1}(a)$ is a single point.  Then $M$ is, up to finite cover, isometric to ${\Bbb S}^n$.
\end{theorem}

\begin{theorem} \label{lem2020-1-7-16}
Let $(g,f)$ be a non-trivial solution of the static vacuum equation on an $n$-dimensional compact manifold $M$ 
  satisfying  $\o = 0$.  
 Then  up to finite cover, either $M$ is isometric to ${\Bbb S}^{n}$ or a warped product ${\Bbb S}^1 \times_\xi {\Sigma}^{n-1}$, where $\Sigma$ is a compact Einstein manifold of positive Ricci curvature.
\end{theorem}
\begin{proof}
By Theorem~\ref{lem2020-9-10-1}, $(M, g)$ is a Bach-flat static vacuum space. In dimension three, we use the fact $C = 0$.
By applying a result due to Qing and Yuan \cite{QY} together with the fact that the scalar curvature is positive,
 $(M, g)$, up to finite cover and  scaling, is isometric to ${\Bbb S}^n$, or 
to the warped product ${\Bbb S}^1 \times_\xi \Sigma$. 
\end{proof}

\begin{remark}\label{rem2024-7-1-1}
{\rm 
In the conclusion of Theorem~\ref{lem2020-1-7-16}, there are two cases by Theorem~\ref{thm2019-12-23-1}:
\begin{itemize}
\item[(i)] $f^{-1}(a)$ is a single point.

In this case, the maximal level set $f^{-1}(b)$ is also a single point,
 and so $(M, g)$ is isometric to ${\Bbb S}^n$ up to a finite cover.
In fact, applying a result in \cite{tas} together with Lemma~\ref{lemt1}, we can see that $M$ is conformal to an $n$-dimensional spherical space.
 
 \item[(ii)]  Both $f^{-1}(a)$ and $f^{-1}(b)$ are hypersurfaces.

In this case, we have   $\a = - \frac{s}{n}$ on $f^{-1}(a)$ and $f^{-1}(b)$ by (\ref{eqn2024-8-26-1}) in Appendix. However, it follows from (\ref{eqn2020-6-28-11-1}) together with
$\n f = 0$ on $f^{-1}(a)$ and $f^{-1}(b)$ that
$$
\a  = - \frac{s}{n} \xi^2
$$
which implies $\xi = 1$ on $f^{-1}(a)$ and $f^{-1}(b)$.
\end{itemize}
}
\end{remark}

In the next section, we will show that $\xi = 1$ on the whole $M$ when  both $f^{-1}(a)$ and $f^{-1}(b)$ are hypersurfaces.

\section{Product static vacuum spaces}

In  this section, we will prove the following.

\begin{theorem} \label{lem2020-1-7-16-2}
Let $(g,f)$ be a non-trivial solution of the static vacuum equation on an $n$-dimensional compact manifold $M$ 
  satisfying  $\o = 0$.  For $\min_{x\in M} f(x) =a$ and $\max_{x\in M} f(x) = b$, if both $f^{-1}(a)$ and $f^{-1}(b)$ are 
  hypersurfaces, then,  up to finite cover,  $M$ is isometric to  a product ${\Bbb S}^1 \times {\Sigma}^{n-1}$, where $\Sigma$ is a compact Einstein manifold of positive Ricci curvature.
\end{theorem}

As mentioned in the end of the previous section, it suffices to show that $\xi \equiv 1$ on $M$.
 Recall  that 
$$
\a = -\frac{s}{n}
$$
on the set $f^{-1}(a)$ and $f^{-1}(b)$ when both $f^{-1}(a)$ and $f^{-1}(b)$ are hypersurfaces. 
Our first observation in this case is  the following.

\begin{lemma}\label{lem2020-7-6-1}
If the function $\dis{\a+\frac{s}{n}}$ has a minimum on $M\setminus \{f^{-1}(a)\cup f^{-1}(b)\}$, then it  attains its minimum on the set $f^{-1}(0)$.
\end{lemma}
\begin{proof}
Recall that $\dis{\a+\frac{s}{n} =0}$ on the set  ${f^{-1}(a)\cup f^{-1}(b)}$. We may assume that $\dis{\a+\frac{s}{n}}$ is not constant (If $\dis{\a+\frac{s}{n}}$ is constant, then it is obvious since $\dis{\a+\frac{s}{n} = 0}$ on $M$).
Suppose  that $\dis{\a + \frac{s}{n}}$ attains its minimum on the set $f^{-1}(c)$ with $a<c<b$.
Then
$$
\a+\frac{s}{n}<0, \quad N(\a) = 0\quad \mbox{and}\quad NN(\a) \ge 0
$$
on the set $f^{-1}(c)$. From 
$$
\dis{f|z|^2 = \n f (\a) - \frac{sf\a}{n-1}}\quad \mbox{or}\quad  \dis{\frac{n}{n-1}f\a^2 = \n f(\a) - \frac{sf\a}{n-1}}
$$
 (recall that $T=0$ and $|z|^2  = \frac{n}{n-1}\a^2$),
we have, on the set $f^{-1}(c)$, 
$$
\frac{nc}{n-1}\a^2 +\frac{c\a s}{n-1} =0, \quad i.e.,\quad \frac{nc}{n-1}\a \left(\a + \frac{s}{n}\right) = 0,
$$
which implies that $c = 0$. 
\end{proof}

\begin{lemma}\label{lem2020-7-2-4}
Let $(g,f)$ be a non-trivial solution of the static vacuum equation  on an $n$-dimensional compact manifold $M$  satisfying  $\o = 0$.  Then, we have $\dis{\a  + \frac{s}{n} \ge 0}$  on $M$.
\end{lemma}
\begin{proof}
By Lemma~\ref{lem2020-7-6-1}, it suffices to show that
$$
\a  + \frac{s}{n} \ge 0
$$
on the set $f^{-1}(0)$.
Suppose that $\dis{\a +\frac{s}{n} <0}$ on $f^{-1}(0)$ which implies that $\dis{{\rm Ric}_M(N, N) = \a +\frac{s}{n} <0}$ on $f^{-1}(0)$.
Since $\Sigma_0:= f^{-1}(0)$ is totally geodesic,   the stability operator for hypersurfaces with vanishing second fundamental form obviously becomes
$$
\int_{\Sigma_0} \left[|\n \vp|^2 - {\rm Ric}(N, N)\vp^2\right] \ge 0
$$
for any function $\vp$ on $\Sigma_0$. By Fredholm alternative (cf. \cite{f-s}, Theorem 1), 
there exists a  positive $\vp>0$ on $\Sigma_0$ satisfying 
$$
\Delta^{\Sigma_0} \vp + {\rm Ric}(N, N)\vp =0.
$$
However,  it follows from the maximum principle
$\vp$ must be a constant which is impossible.
\end{proof}

\begin{lemma}[cf. \cite{Be}] \label{lem2020-7-2-1}
Let $(M^n, g)$ be a Riemannian manifold with constant scalar curvature $s$. Then
\bea
\d d^D z = D^*Dz + \frac{n}{n-2}z\circ z + \frac{s}{n-1}z - {\mathring {\mathcal W}}z - \frac{1}{n-2}|z|^2 g.
\eea
\end{lemma}

\begin{lemma}\label{cor191}
Let $(g,f)$ be a non-trivial solution of the static vacuum equation  on an $n$-dimensional compact manifold $M$  satisfying  $\o = 0$.   Suppose that  $f^{-1}(a)$ is a hypersurface. If $\dis{\a  + \frac{s}{n} \ge 0}$ on $M$, then
$$
\a + \frac{s}{n} =0
$$
on the whole $M$.
\end{lemma}
\begin{proof}
We have $T=0, C=0$ and so ${\mathring {\mathcal W}}z = 0$.
So, it follows from  Lemma~\ref{lem2020-7-2-1} that
$$
-\langle D^*Dz, z\rangle = \frac{s}{n-1}|z|^2 + \frac{n}{n-2}\langle z\circ z, z\rangle.
$$
Thus
\be
\frac{1}{2}\Delta |z|^2 = -\langle D^*Dz, z\rangle  + |Dz|^2 = \frac{s}{n-1}|z|^2 + \frac{n}{n-2}\langle z\circ z, z\rangle +|Dz|^2.\label{eqn2020-7-2-2}
\ee
Since $T = 0$, we have
$$
z(N, N) = \a, \quad z(E_i, E_i) = - \frac{\a}{n-1}
$$
and so
$$
|z|^2 = \frac{n}{n-1}\a^2, \quad \langle z\circ z, z\rangle = \a^3 - \frac{1}{(n-1)^2}\a^3 .
$$
Substituting these into (\ref{eqn2020-7-2-2}), we obtain
\bea
\frac{1}{2}\Delta |z|^2  
&=&
 \left(\frac{n}{n-1}\right)^2 \a^2 \left(\frac{s}{n}+ \a\right) +|Dz|^2.
 \eea
 Since $|z|^2$ is a subharmonic function, it is  constant and so is $\a$ because $\dis{|z|^2 = \frac{n}{n-1}\a^2}$.
 Since $\dis{\a = -\frac{s}{n}}$ on $f^{-1}(a)$, we have
 $$
 \a = - \frac{s}{n}
 $$
 on the whole $M$.
\end{proof}


\begin{lemma}\label{lem2020-7-2-4-1}
Let $(g,f)$ be a non-trivial solution of the static vacuum equation  on an $n$-dimensional compact manifold $M$  satisfying  $\o = 0$. 
If both $f^{-1}(a)$ and $f^{-1}(b)$ are hypersurfaces, then
$$
\xi \equiv 1
$$
on the whole $M$.
\end{lemma}
\begin{proof}
First, we claim that $\xi =1$ on the set $f^{-1}(0)$. By  Lemma~\ref{lem2020-7-2-4} and Lemma~\ref{cor191}, we have
$$
{\rm Ric}_M(N, N) = \a + \frac{s}{n} = 0\quad \mbox{on $M$}.
$$
Since $f^{-1}(0)$ is totally geodesic, we have 
$$
s_0 = s - 2{\rm Ric}_M(N, N) = s, 
$$
where $s_0$ is the scalar curvature of $f^{-1}(0)$ with respect to the induced metric.
Next, it follows from (\ref{eqn2020-6-28-10}) that
$$
s = s_0 - {2(n-1)}\frac{\xi''}{\xi}  - (n-1)(n-2)\frac{\xi'^2}{\xi^2},
$$
where  $\xi'$ and $\xi''$ denote the derivatives of $\xi$ with respect to $f$ (cf. \cite{Be}). Thus, we obtain
\bea
 {2(n-1)}\frac{\xi''}{\xi}  - (n-1)(n-2)\frac{\xi'^2}{\xi^2}  = 0 
\eea
on the set $f^{-1}(0)$.
Now from (\ref{eqn2020-6-28-10}) again, it is easy to see that
$$
{\rm Ric}_M (N, N) = - (n-1)\frac{\xi''}{\xi}
$$
on the set $f^{-1}(0)$. Consequently, we have
$$
{\xi''} =0, \quad \xi' = 0
$$
on the set $f^{-1}(0)$.
From (\ref{eqn2020-6-28-11-1}), we have
$$
\a + \frac{s}{n}\xi^2 = -(n-1)|\n f|^2\xi  \cdot \frac{\xi'}{f}
$$
on the set $M \setminus f^{-1}(0)$. By letting $x \to p \in f^{-1}(0)$ and applying L'Hospital rule, we obtain
\bea
\a + \frac{s}{n}\xi^2 =- (n-1)|\n f|^2 \xi \xi'' =0.
\eea
So, on the set $f^{-1}(0)$, we have
$$
\a + \frac{s}{n}\xi^2  =0 = \a + \frac{s}{n},
$$
which implies that $\xi = 1$ on the set $f^{-1}(0)$.

Now we claim that  $\xi = 1$ on the whole $M$. Since $\dis{\a = - \frac{s}{n}}$ on $M$, we have, from (\ref{eqn2020-6-28-11-1}), 
\be
(n-1)\xi|\nabla f|^2 \xi'= \frac{sf}{n}(1-\xi^2). \label{eqn2020-7-2-10}
\ee
Taking the derivative with respect to $f$, we have
$$
{\mathcal F}(\xi'', \xi') = \frac{s}{n}(1 - \xi^2),
$$
where
$$
{\mathcal F}(\xi'', \xi') = (n-1)|\n f|^2 \xi \xi'' + (n-1)|\n f|^2 (\xi')^2 + (n-1)(|\n f|^2)' \xi \xi' + \frac{2sf}{n}\xi \xi'.
$$
From (\ref{eqn2020-7-2-10}) again, we have
$$
f {\mathcal F}(\xi'', \xi') = \frac{sf}{n}(1 - \xi^2) = (n-1)|\n f|^2 \xi \xi',
$$
i.e.,
\be
{\mathcal F}(\xi'', \xi')-  (n-1)\frac{|\n f|^2 \xi \xi'}{f} =0.\label{eqn2020-7-2-11}
\ee
Applying the maximum principle to $\xi$ in (\ref{eqn2020-7-2-11}) on the set $a < f<0$, we obtain
$$
\sup_{a<f<0} \xi = \sup_{\{f=a\}\cup \{f=0\}} \xi = 1
$$
and 
$$
\inf_{a<f<0} \xi = \inf_{\{f=a\}\cup \{f=0\}} \xi = 1.
$$
Hence $\xi = 1$ on the set $a\le f \le 0$. The same argument shows that
$\xi = 1$ on the set $0 \le f \le b$. Consequently $\xi = 1$ on $M$.
\end{proof}

\section{compact static vacuum spaces with PIC}

In this section, we assume that $(M^n, g,f)$, $n \ge 4$,  is a compact  static vacuum space having positive isotropic curvature such that $\o = df\wedge i_{\n f}z =0$.

\begin{lemma}\label{thm2020-6-18-1}
Let  ${\min_M f = a}$ and suppose that $f^{-1}(a)$ is a hypersurface.  Then  $f^{-1}(a)$ is homeomorphic  to ${\Bbb S}^{n-1}$, up to finite cover.
\end{lemma}
\begin{proof}
Let $\Sigma = f^{-1}(a)$. It follows  from  Theorem~\ref{thm2019-12-23-1} and Theorem~\ref{lem2020-1-7-16}  that $\Sigma$ is a connected, totally geodesic stable minimal hypersurface of $M$ and Einstein with positive  Ricci curvature.
Thus $\Sigma$ also  has positive isotropic curvature with respect to the induced metric for $n \ge 5$.
In particular, since the fundamental group of $\Sigma$ is finite,  $\Sigma$ turns out to be homeomorphic to ${\Bbb S}^{n-1}$ up to finite cover \cite{mm88}.

When $n=4$, Chen-Tang-Zhu \cite{ctz} gave a complete classification on compact $4$-manifolds with PIC. In our case that $f^{-1}(a)$ is hypersurface, since $M$ satisfies the static vacuum equation, $({\Bbb R}\times {\Bbb S}^3)/\Gamma$ is the only one possible choice, 
where $\Gamma$ is a cocompact fixed point free discrete isometric subgroup of the standard ${\Bbb R}\times {\Bbb S}^3$. This completes our proof.
\end{proof}

Combining Theorem~\ref{lem2020-1-7-16}, Lemma~\ref{lem2020-7-2-4-1}  and Lemma~\ref{thm2020-6-18-1}, we obtain the following which is one of our main results.

\begin{theorem}\label{cor2023-4-22-1}
Let $(M^n, g, f), n \ge 4, $  be a $n$-dimensional compact static vacuum space of positive isotropic curvature such that $\o = df\wedge i_{\n f}z =0$. Then $M$ is isometric to a sphere ${\Bbb S}^n$  or a  product ${\Bbb S}^1\times {\Bbb S}^{n-1}$,  up to finite cover.
\end{theorem}

In dimension $3$, we have the  same rigidity as the result holds for $n \ge 4$.

\begin{corollary}\label{cor2023-4-18-1}
Let $(M^3, g, f)$  be a $3$-dimensional compact static vacuum space of positive isotropic curvature such that 
$\o = df\wedge i_{\n f}z =0$. Then $M$ is isometric to a sphere ${\Bbb S}^3$  or a  product ${\Bbb S}^1\times {\Bbb S}^{2}$,  up to finite cover.
\end{corollary}
\begin{proof}
In view of  Theorem~\ref{thm2020-6-18} and Theorem~\ref{lem2020-1-7-16}, we can see that $\Sigma^2$ is a compact Einstein manifold with positive scalar curvature with finite fundamental group. So it is well-known that  $\Sigma$ is just ${\Bbb S}^2$, up to finite cover.
\end{proof}


\section{Appendix}

In this section, we will give a proof for Theorem~\ref{thm2019-12-23-1}.

\vspace{.12in}
\noindent
{\bf Proof of Theorem~\ref{thm2019-12-23-1}.}
Let ${\min_M f = a}$. We may assume $a<0$ from the maximum principle together with
(\ref{eqn2020-1-14}).
By Lemma~\ref{cor191-1} and Morse theory, the inverse set
$f^{-1}(a)$ is a single point, or a hypersurface of $M$ which is homotopically equivalent to a connected component of $f^{-1}(0)$.
In fact, the set $f^{-1}(a)$ consists of more than two points, there should be a critical point of $f$ which is not a minimum point nor maximum point.
By the same reason, if $f^{-1}(a)$ is a single point, then so is $f^{-1}(b)$ and vice versa. 
In fact, if $f^{-1}(a)$ is a single point and $f^{-1}(b)$ is a hypersurface, then we can choose a connected component $\Omega$ whose boundary is $f^{-1}(b)$. Applying the maximum principle to (\ref{eqn2020-1-14}) on the set $\Omega$, we get a contradiction.

Suppose that $f^{-1}(a)$ is a hypersurface. Then 
 for a sufficiently small $\epsilon >0$, $f^{-1}(a+\epsilon)$ has two
  connected components, say $\Sigma_\epsilon^+, \,\, \Sigma_\epsilon^-$.
  Note that $\Sigma_0^+=\Sigma_0^- = f^{-1}(a)$.
    Let $\nu$ be a unit normal vector field on $\Sigma:= f^{-1}(a)$. Then 
  $\nu$ can be extended smoothly to a vector field $\Xi$  defined  on a a tubular neighborhood 
  of $f^{-1}(a)$ such that   $\Xi|_{f^{-1}(a)} = \nu$ and 
  $\Xi|_{\Sigma_\epsilon^+} =  N = \frac{\n f}{|\n f|}$, 
  $\Xi|_{\Sigma_\epsilon^-} =  -N = -\frac{\n f}{|\n f|}$.
 Note that
 $$
 \lim_{\epsilon\to 0+} N = \lim_{\epsilon \to 0-}(-N)  = \nu.
 $$
   Next, on the hypersurface $f^{-1}(a+\epsilon)$ near $f^{-1}(a)$, the Laplacian of $f$ is given by
  $$
  \Delta f = \Delta' f + Ddf(N, N) + m \langle N, \n f\rangle = Ddf(N, N) + m \langle N, \n f\rangle,
  $$
  where  $\Delta'$ and $m$ denote the Laplacian and mean curvature of $f^{-1}(a+\epsilon)$,
  respectively. In particular, on the minimum set $f^{-1}(a)$, we have
  \be
  \Delta f =   Ddf(\nu, \nu). \label{eqn2020-1-7-15}
  \ee
Let $p \in f^{-1}(a)$ be any point.  Choosing an orthonormal basis $\{e_1 = \nu(p), e_2, \cdots, e_n\}$ on $T_pM$,
 it follows from (\ref{eqn2020-1-7-15}) together with $a = \min_M f$  that
\be
f z_p(e_i, e_i) - \frac{sf}{n(n-1)}  = Ddf_p(e_i, e_i)  = 0\label{eqn2020-6-5-1}
\ee
for each $i =2, \cdots, n$. This also shows that, on the minimal set $f^{-1}(a)$, 
\be
\lim_{\e\to 0} z(N, N) = z(\nu, \nu) = - \frac{s}{n}.\label{eqn2024-8-26-1}
\ee
Now we claim that  the minimum set  $\Sigma := f^{-1}(a)$ is totally geodesic,  and in particular the mean curvature is vanishing, $m=0$.
 In fact, fix $i$ $(i=2,3, \cdots, n)$, say $i=2$,  and  let $\gamma : [0, l) \to M$ be a unit speed geodesic such that $\gamma(0) =p\in \Sigma$
and $\gamma'(0) = e_2 \in T_p\Sigma$. Then, on a tubular neighborhood of $\Sigma$,  we have 
\be
D_{\gamma'} N &=&
\langle \gamma', N\rangle N\left(\frac{1}{|\n f|}\right)\n f + \frac{1}{|\n f|}\left[f z(\gamma', \cdot) - \frac{sf}{n(n-1)}\gamma'\right].
\label{eqn2020-6-17-1}
\ee
Note  that
$$
N\left(\frac{1}{|\n f|}\right)  = - \frac{1}{|\n f|^2} \left(f\alpha - \frac{sf}{n(n-1)}\right).
$$
Letting
$$
\gamma' = \langle \gamma', N\rangle N + \gamma'^{\top},
$$
where $\gamma'(t)^{\top}$ is the tangential component of $\gamma'(t)$ to $f^{-1}(\gamma(t))$, 
and substituting these into (\ref{eqn2020-6-17-1}), we obtain
$$
|\n f| D_{\gamma'} N =   f z(\gamma'^\top, \cdot)  - \frac{sf}{n(n-1)} \gamma'^\top.
$$
Taking the covariant derivative in the direction $N$, we have
\bea
&&Ddf(N, N) D_{\gamma'}N + |\n f|D_ND_{\gamma'}N \\
&&\qquad
= 
|\n f| z(\gamma'^\top, \cdot)  + f D_N[z(\gamma'^{\top}, \cdot)] -  \frac{s}{n(n-1)}|\n f| \gamma'^\top - \frac{sf}{n(n-1)}D_N \gamma'^{\top}.
\eea
Letting $t\to 0+$, we obtain
\be
-\frac{sa}{n-1} D_{e_2}\nu = a D_{\nu} [z(\gamma'^\top, \cdot)]\bigg|_p - \frac{sa}{n(n-1)} D_{\nu} \gamma'^\top \bigg|_p \label{eqn2021-1-25-2}
\ee
because the covariant derivative depends only on the point $p$ and initial vector $e_2$.
Now since $z(N, X) = 0$ for $X \perp N$, we may assume that $\{e_i\}_{i=2}^n$ diagonalizes $z$ at the point $p$ so that
\be
z(D_{\nu} \gamma'^\top|_p, e_2) = 0.\label{eqn2021-1-25-3}
\ee

\vspace{.12in}

{\bf Assertion:}\, $\nu(z(\gamma'(t), \gamma'(t))|_{t=0} = 0.$

\vspace{.15in}
\noindent
Define $\vp(t):= f\circ \gamma(t)$. Note that $\vp'(0) =0$ and also $\vp''(0) = Ddf(e_2, e_2) = 0$ by (\ref{eqn2020-6-5-1}).
Since $\Sigma$ is the minimum set of $f$, $\vp'(t)$ is nondecreasing when $\vp(t)$ is sufficiently close to $a = \min f$ and
so $\vp''(t) \ge 0$ for sufficiently small $t>0$. So
\bea
\vp''(t) = Ddf(\gamma'(t), \gamma'(t)) = \vp(t)  z(\gamma'(t), \gamma'(t)) - \frac{s}{n(n-1)}\vp(t) \ge 0 
\eea
for sufficiently small $t>0$.  However, by (\ref{eqn2020-6-5-1}), we have $z(\gamma'(t), \gamma'(t)) >0$ for sufficiently small $t>0$.
So,
$$
0< z(\gamma'(t), \gamma'(t)) \le \frac{s}{n(n-1)}
$$
for sufficiently small $t$. Defining $\xi(t) = z(\gamma'(t), \gamma'(t)) - \frac{s}{n(n-1)}$, we have
$\xi(t) \le 0$ and $\xi(0) = 0$ by  (\ref{eqn2020-6-5-1}). Thus,
$$
\xi'(0)  = \frac{d}{dt}\bigg|_{t=0} z(\gamma'(t), \gamma'(t)) \le 0.
$$
Now considering a smooth extension of $\gamma$ to the interval $(-\e, 0]$, 
we can see that  $\frac{d}{dt}|_{t=0} z(\gamma'(t), \gamma'(t))$ cannot be negative since $\Sigma$ is the minimum set of $f$. 
In other words, we must have
$$
\frac{d}{dt}\bigg|_{t=0} z(\gamma'(t), \gamma'(t)) = 0,
$$
which proves our assertion.

\vspace{.15in}
\noindent
Let $\{N, E_2, \cdots, E_n\}$ be a local frame around $p$ such that $E_i(p) = e_i$ for $2 \le i$
and $E_2 = \frac{\gamma'^\top}{|\gamma'^\top|}$. Then, for $i \ge 2$,
$$
z(E_i, \cdot) = \sum_{j=2}^n z(E_i, E_j) E_j
$$
and
$$
D_N[z(\gamma'^\top, \cdot)] = D_N \left[z(\gamma'^\top, E_j)E_j\right] = N\left(z(\gamma'^\top, E_j)\right)E_j + z(\gamma'^\top, E_j)D_NE_j.
$$
So,
$$
\langle D_N[z(\gamma'^\top, \cdot)], E_2\rangle  = N\left(z(\gamma'^\top, E_2)\right) + z(\gamma'^\top, E_j) \langle D_NE_j, E_2\rangle.
$$
Note that $z(\gamma', \gamma') = |\gamma'^\top| z(\gamma'^\top, E_2)
 + \langle \gamma, N\rangle^2 \a$ with $\a = z(N, N)$.
Since $|\gamma'^\top|$ attains its maximum at $p$ and $\langle \gamma', N\rangle^2$ attains its minimum at $p$, we have
$$
\nu [z(\gamma', \gamma')]\bigg|_p = |\gamma'^\top|  \nu[z(\gamma'^\top, E_2)]\bigg|_p 
+ \langle \gamma, N\rangle^2 \nu(\a)(p) =  \nu[z(\gamma'^\top, E_2)]\bigg|_p,
$$
which shows that
\be
\nu[z(\gamma'^\top, E_2)]\bigg|_p = 0\label{eqn2021-2-8-1}
\ee
by {\bf Assertion}. Letting $t\to 0+$ and applying  (\ref{eqn2021-2-8-1}),  we obtain
\bea
\langle D_\nu [z(\gamma'^\top, \cdot)], E_2\rangle|_p  &=& \nu\left(z(\gamma'^\top, E_2)\right)|_{t=0} + z(e_2, e_2) \langle D_\nu E_2, E_2\rangle|_p\\
&=&0.
\eea
Thus, by  (\ref{eqn2021-1-25-2}) and (\ref{eqn2021-1-25-3}) again, we have
$$
- \frac{sa}{n-1} \langle D_{e_2}\nu, e_2\rangle = a \langle D_{\nu}[z(\gamma'^\top, \cdot)], E_2\rangle|_p - \frac{sa}{n(n-1)} \langle D_{\nu} E_2, E_2\rangle|_p =0.
$$
Since $i=2$ is, in fact, arbitrary, we have 
$$
\langle D_{e_i}\nu, e_i\rangle = 0
$$
for any $i \ge 2$. Hence,  the square norm of the second fundamental form $A$ is given by
$$
|A|^2(p) = \sum_{i=2}^n  \left|\left(D_{e_i}e_i\right)^\perp\right|^2 = \sum_{i=2}^n \langle D_{e_i}\nu, e_i\rangle^2 = 0,
$$
which shows $f^{-1}(a)$ is totally geodesic.

Finally, from (\ref{eqn2020-1-7-15}) together with the static vacuum equation (\ref{eqn2020-1-7-11}), we have
 $$
  -\frac{sf}{n-1} = Ddf(\nu, \nu) = fz(\nu, \nu) - \frac{sf}{n(n-1)}.
  $$
That is, 
\bea
 z_p(\nu, \nu) = - \frac{sf}{n f} = - \frac{s}{n} 
  \eea
on the set $\Sigma = f^{-1}(a)$ as mentioned above. Thus, 
$$
r(\nu, \nu) = {\rm Ric}(\nu, \nu) =   0
$$
and obviously the stability operator for hypersurfaces with vanishing second fundamental form becomes
$$
\int_{\Sigma} |\n \vp|^2  \ge 0
$$
for any function $\vp$ on $\Sigma$. 
Consequently, $\Sigma = f^{-1}(a)$ is a totally geodesic stable minimal hypersurface of $M$.
The very similar argument shows that the same property holds for
$f^{-1}(b)$ with $\max_M f = b$.
Final arguments  follow from Lemma~\ref{cor191-1}. \hfill{\Large$\Box$}

\vspace{.12in}
\noindent
{\bf Acknowledgement} 
The first and second authors were supported by the Basic Science Research Program through the National Research Foundation of Korea(NRF)  funded by the Ministry of Education (NRF-2018R1D1A1B05042186) and (RS-2024-00334917), respectively.


\begin{thebibliography}{99}


\bibitem{bach} R. Bach, {\it Zur Weylschen Relativit\"atstheorie und
der Weylschen Erweiterung des Kr\"ummungstensorbegriffs},
Math. Z.  {\bf 9} (1921), 110--1359.



\bibitem{Be} A. L. Besse, Einstein Manifolds, New York: Springer-Verlag 1987


\bibitem{Bour} J. P. Bourguignon, {\it Une stratification de l'espace des structures riemanniennes}, { Compositio Math.} {\bf 30} (1975), no. 1,  1--41.


\bibitem{bren} S. Brendle,  {\it Einstein manifolds with nonnegative isotropic curvature are locally symmetric}, Duke Math. J.  {\bf 151} (2010), no. 1, 1--21.



\bibitem{CC} H.D. Cao and Q. Chen, {\it On Bach-flat gradient shrinking Ricci solitons}, Duke Math. J. {\bf 162} (2013), no.6, 1149--1169.



\bibitem{c-hu} B-L. Chen, X-T. Huang, {\it Four-manifolds with  positive isotropic curvature}, Front. Math. China {\bf 11} (5) (2016),  1123--1149.

\bibitem{ctz} B-L. Chen, S-H. Tang, X-P. Zhu, {\it Complete classification of compact four-manifolds with positive isotropic curvature}, J. Diff. Geom. {\bf 91} (2012), 41--80.





\bibitem{corv} J. Corvino, {\it Scalar curvature deformation and a gluing construction for the Einstein constraint equations}, Commun. Math. Phys. {\bf 214} (2000), 137--189.


\bibitem{eji} N. Ejiri, 
A negative answer to a conjecture of conformal transformations of Riemannian manifolds,  J. Math. Soc. Japan
{\bf 33} (1981), 261--266.


\bibitem{f-m} A. E. Fischer and J. E. Marsden,
Manifolds of Riemannian metrics with prescribed scalar curvature,
Bull. Amer. Math. Soc. {\bf 80} (1974), 479--484.

\bibitem{f-s}  D. Fischer-Colbrie, R. Schoen, 
{\it The Structure of Complete Stable Minimal Surfaces in
$3$-Manifolds of Non-Negative Scalar Curvature}, 
Comm. Pure and App. Math. {\bf  33} (1980), 199--211.

\bibitem{fra} A. Fraser, {\it Fundamental groups of manifolds with
 positive isotropic curvature}, Ann.  Math.   {\bf 158} (1) (2003),  345--354.

\bibitem{f-w} A. Fraser,  J. Wolfson, {\it The fundamental group of manifolds of positive isotropic curvature and surface groups}, 
Duke Math. J.  {\bf 133} (2) (2006),  325--334.



\bibitem{h-y} S. Hwang, G. Yun, {\it static vacuum spaces with vanishing of complete divergence of Weyl tensor}, J.  Geom. Anal. {\bf 31} (2021), no. 3, 3060--3084.


\bibitem{tas2} S. Ishihara, Y. Tashiro, {\it On Riemannain manifolds admitting a concircular transformation}, Math. J.  Okayama Univ. {\bf 9}, 19--47.


\bibitem{kob} O. Kobayashi, 
A differential equation arising from scalar curvature function, J. Math. Soc. Japan 
{\bf 34} (1982) (4), 665--675.

\bibitem{ks} J. Kim, J. Shin, {\it Four-dimensional static and related critical spaces with harmonic curvature}, Pacific J. Math. {\bf 295} (2) (2018), 429--462.  

\bibitem{k-o} O. Kobayashi and M.  Obata, 
Conformally-flatness and static space-times, Manifolds and Lie groups, Progress in Mathematics, {\bf 14}, Birkh\"auser, (1981) 197-206.

\bibitem{lab} M.-L. Labbi, {\it On compact manifolds with  positive isotropic curvature}, 
Proc.  Amer. Math. Soc.  {\bf 128} (5) (1999),  1467--1474.

\bibitem{laf} J. Lafontaine, {\it Sur la g\'eom\'etrie d'une g\'en\'eralisation de l'\'equation diff\'erentielle d'Obata},
J. Math. Pures Appl. {\bf 62}, no 1 (1983) 63-72.

\bibitem{lis} M. Listing, {\it Conformally invariant Cotton and Bach tensor in $N$-dimensions}, arXive:math/0408224v1 [math.DG] 17 Aug 2004.

\bibitem{pl1} P. Li, Geometric Analysis, Cambridge studies in advanced mathematics, 134, Cambridge University Press.

\bibitem{mm88} M. Micallef, J. D. Moore, {\it Minimal two-spheres and the topology of 
manifolds with positive curvature on totally isotropic two-planes}, Ann.  Math. 
{\bf 127}  (2) (1988),  199--227.


\bibitem{m-w} M. Micallef, M. Y. Wang, {\it Metrics with nonnegative 
isotropic curvature}, Duke Math. J.  {\bf 72} (3) (1993),   649--672.

\bibitem{shen} Y. Shen, A note on Fischer-Marsden's conjecture, Proc. Amer. Math. Soc. 
 {\bf 125} (1997),  901--905.
 
\bibitem{QY} J. Qing, W. Yuan, {\it  A note on static spaces and related problems}, J. Geom. Phys. {\bf 74} (2013), 13--27.

\bibitem{sea} W. Seaman, {\it On manifolds with nonnegative curvature on totally isotropic $2$-planes}, Trans. of Amer. Math. Sco.  {\bf 338} (2) (1993),   843--855.

\bibitem{ses} H. Seshadri,  {\it Isotropic Curvature: A Survey}, 
S\'eminaire de th\'eorie spectrale et g\'eom\'etrie, Grenoble,
{\bf 26} (2007-2008), 139--144.



\bibitem{tas} Y. Tashiro, {\it Complete Riemannian manifolds and some vector fields}, Trans. Amer. Math. Soc. {\bf 117} (1965), 251--275.

\bibitem{wu} H. Wu, The Bochner Technique in Differential Geometry, in; Mathematical Reports, vol. 3, Pt 2, Harwood Academic Publishing, London, 1987.





 
 





\bibitem{ych} G. Yun, J. Chang, and S. Hwang, {\it Total scalar curvature and harmonic curvature}, Taiwanese J. Math. {\bf 18} (2014), no.5, 1439--1458. 


\bibitem{GH1} G. Yun, S. Hwang, {\it Gap Theorems on critical point equation 
of the total scalar curvature with divergence-free Bach tensor}, 
Taiwanese J. Math. {\bf 23} (4) (2019), 841--855.


\end{thebibliography}
\end{document}